\documentclass{amsart}
\usepackage{amsmath,amsthm,amscd,amssymb}
\usepackage{latexsym}
\usepackage{txfonts}

\numberwithin{equation}{section}

\theoremstyle{plain}
\newtheorem{theorem}[equation]{Theorem}
\newtheorem{lemma}[equation]{Lemma}

\newtheorem{proposition}[equation]{Proposition}

\theoremstyle{definition}
\newtheorem{definition}[equation]{Definition}

\theoremstyle{remark}

\newcommand{\supp}{\operatorname{supp}}
\newcommand{\dist}{\operatorname{dist}}
\newcommand{\loc}{\operatorname{loc}}

\newcommand{\tr}{\operatorname{tr}}
\newcommand{\bg}{\operatorname{big}}
\newcommand{\buffer}{\operatorname{buffer}}
\newcommand{\bad}{\operatorname{Bad}}
\newcommand{\good}{\operatorname{Good}}
\newcommand{\I}{\operatorname{I}}
\newcommand{\II}{\operatorname{II}}
\newcommand{\III}{\operatorname{III}}
\newcommand{\IV}{\operatorname{IV}}

\newcommand{\error}{\operatorname{Error}}

\begin{document}

\title{A proof of the local $Tb$ Theorem for standard
Calder\'{o}n-Zygmund operators}

\author[S. Hofmann]{Steve Hofmann}
\address{Department of Mathematics, 
University of Missouri, Columbia, Missouri 65211, USA}
\email{hofmann@math.missouri.edu}
\thanks{The author was supported by the National Science Foundation}

\begin{abstract} 
We give a proof of a so-called ``local $Tb$" Theorem for singular integrals whose kernels satisfy the
standard Calder\'on-Zygmund conditions.  The present theorem, which extends an earlier result of
M. Christ \cite{Ch}, was proved in \cite{AHMTT} for ``perfect dyadic" Calder\'on-Zygmund operators.
The proof in \cite{AHMTT} essentially carries over to the case considered here, with some technical adjustments.
\end{abstract}

\maketitle

\section{Introduction\label{s1}}

Following Coifman and Meyer, we say that an operator $T$, initially 
defined as a mapping from test functions $C_0^\infty(\mathbb{R}^n)$ to distributions, is a singular integral operator if
it is associated to a kernel $K(x,y)$ in the sense that for all 
$\phi,\psi \in C_0^\infty$ with disjoint supports, we have
$$\langle T\phi,\psi\rangle = \iint_{\mathbb{R}^n \times \mathbb{R}^n}K(x,y)\phi(y)\psi(x)dydx,$$
and if the kernel satisfies the standard ``Calder\'on-Zygmund" bounds
 \begin{subequations}\label{eq6.3}
\begin{equation}\label{eq6.3a}|K(x,y)|\leq \frac{C}{|x-y|^n}\end{equation}
\begin{equation}\label{eq6.3b} |K(x,y+h)-K(x,y)|+|K(x+h,y)-K(x,y)|\leq C\frac{|h|^\alpha}{|x-y|^{n+\alpha}},\end{equation}\end{subequations}
where the later inequality holds for some $\alpha >0$ whenever $|x-y|>2 |h|$.

For future reference,  we note that, for any kernel $K(x,y)$ satisfying \eqref{eq6.3}(a), and for
$1<p<\infty$, we have
\begin{equation}\label{eq6.13} \int_Q \left| \int K(x,y)1_{6Q\backslash Q}(y)f(y)dy\right|^p dx\leq C_p\int_{6Q\backslash
Q} |f|^p.\end{equation}
We omit the proof.

The following theorem is an extension of a local Tb Theorem for singular
integrals introduced by M. Christ \cite{Ch} in connection with the theory of analytic capacity. 
See also \cite{NTV}, where a non-doubling versions of Christ's local $Tb$ Theorem
is given.  A 1-dimensional
version of the present result, valid for ``perfect dyadic" Calder\'on-Zygmund kernels, appears in \cite{AHMTT}.  In the sequel, we use the notation $T^{tr}$ to denote the transpose of the operator
$T$.

\begin{theorem}\label{t6.6} Let $T$ be a singular integral operator associated to a kernel $K$ satisfying
\eqref{eq6.3}, and suppose that $K$ satisfies the generalized truncation condition $K(x,y)\in L^\infty
(\mathbb{R}^n\times \mathbb{R}^n)$. Suppose also that there exist pseudo-accretive systems $\{ b^1_Q\}$, $\{ b^2_Q\}$
such that $b^1_Q$ and  $b^2_Q$ are supported in $Q$, and
\begin{enumerate}\item[(i)] \qquad$\int_Q \left(|b^1_Q|^q+|b^2_Q|^q\right)\leq C|Q|$, for some $q>2$
\item[(ii)] \qquad$\int_Q \left(|Tb^1_Q|^2 + |T^{\tr} b^2_Q |^2\right)\leq C|Q|$
\item[(iii)]\qquad $\frac{1}{C}|Q|\leq \min\left(\Re e \int_Q b^1_Q,
\Re e \int_Qb^2_Q\right)$.\end{enumerate}
Then $T:L^2(\mathbb{R}^n)\to L^2(\mathbb{R}^n)$, with bound independent of $\| K\|_\infty $.\end{theorem}

The theorem in \cite{Ch} is similar, except that the $L^2$ (or $L^{2+\epsilon}$) control in
conditions $(i)$ and $(ii)$ is replaced by $L^\infty$ control.
The proof of the present theorem follows that of \cite{AHMTT}, except for some technical adjustments related to the presence of
the Calder\'on-Zygmund tails in condition \eqref{eq6.3b}.  These tails do not appear in
the perfect dyadic setting considered in \cite{AHMTT}, and their absence allows
one to take $q=2$ in condition $(i)$;  moreover, Auscher and Yang \cite{AY}
have extended the present result to the case $q=2$, by reducing to \cite {AHMTT}. At present,
we do not know a direct proof of our theorem without taking $q>2$, nor 
(in contrast to the perfect dyadic case) any proof with $q<2$.

The present version of the theorem has been applied in \cite{AAAHK} to establish $L^2$
boundedness
of  layer potentials associated to certain divergence form elliptic operators with bounded measurable coefficients.

\section{Preliminaries \label{s2}}

We begin by setting some notation, and recalling some familiar facts. 
In particular, we discuss adapted averages and difference operators following
\cite{CJS}.  We define the standard dyadic conditional
expectation and martingale difference operators
$$E_kf(x)=\sum_{Q\in \mathbb{D}_k}1_Q(x)\,\frac{1}{|Q|}\int_Qf,$$
where $\mathbb{D}_k$, $k\in \mathbb{Z}$, denotes the standard grid of dyadic cubes in $\mathbb{R}^n$ having side
length $2^{-k}$, and
\begin{align*} \Delta_k&\equiv E_{k+1}-E_k.\\\intertext{Then} E_jE_k&=E_k,\quad j\geq k\end{align*}
and thus also
\begin{equation}\begin{split}\label{eq8.1} \Delta_j\Delta _k &=0,\quad j\neq k\\\Delta_k^2&=\Delta
_k\end{split}\end{equation}
Moreover, the operators $E_k$ and $\Delta _k$ are self-adjoint. Consequently, we have the square function identity
\begin{equation}\label{eq8.2} \int_{\mathbb{R}^n}\sum^\infty_{k=-\infty} |\Delta _kf|^2=\| f\|^2_2,\end{equation}
as well as the discrete Calder\'on reproducing formula
\begin{equation}\label{eq8.3} \sum \Delta^2_k=\sum \Delta_k =I,\end{equation}
where the convergence is in the strong operator topology on $L^2$, as well as point-wise $a.e.$ for $f\in L^2$, as may be
seen by the telescoping nature of the sum, and the fact that
\begin{equation}\label{eq8.4} \lim_{k\to \infty} E_kf=f \text{ a.e.},\quad f\in L^p_{\loc} \, , \, 1\leq p\leq
\infty\end{equation}
(by Lebesque's Differentiation Theorem), and 
\begin{equation}\label{eq8.5} \lim_{k\to -\infty} E_kf=0,\quad f\in L^p,1\leq p<\infty.\end{equation}
Details may be found in \cite{St}. As a consequence of \eqref{eq8.2}, we have the standard dyadic Carleson measure
estimate.

\begin{proposition}\label{p8.6} There exists a constant $C$ such that for every dyadic cube $Q$,
$$\frac{1}{|Q|} \int_Q \sum_{k:2^{-k}\leq \ell (Q)} |\Delta_k h(x)|^2dx\leq C\| h\|_{BMO}^2.$$  \end{proposition}
{\em Remark.}  The well-known proof is the same as that in the continuous parameter case \cite{FS}, and is omitted.

Suppose now that $b$ is dyadically pseudo-accretive, i.e.
\begin{equation}\label{DPA}\tag{D$\psi$A} b\in L^\infty,\quad |E_kb|\geq \delta,\end{equation}
for some $\delta>0$, and for all $k\in \mathbb{Z}$, or more generally that
\begin{equation}\label{eq8.7} \left| \frac{1}{|Q|}\int_Q b\right|\geq \delta \,,\quad \int_Q |b |^2 \leq C |Q|\end{equation}
for all $Q$ in some ``good" subset of $\mathbb{D}_k$. Then we can define the adapted expectation operators
$$E^b_kf=\frac{E_k(fb)}{E_k(b)}$$
(at least on the good cubes), and we can also define the martingale difference operators
$$\Delta_k^b=E^b_{k+1}-E^b_k,$$
at least on cubes $Q\in \mathbb{D}_{k}$ which are not only ``good", but whose dyadic children are also ``good"  (in the sense of \eqref{eq8.7}).
The following result is well known (see, e.g. [Ch2, p. 45])

\begin{proposition}\label{p8.8} Suppose $b\in D\psi A$. Then we have the following square function estimate
$$\int_{\mathbb{R}^n}\sum |\Delta^b_kf|^2\leq C\|f\|^2_2.$$\end{proposition}
We omit the proof.

It is routine to check that for $b\in D\Psi A$, $E^b_k$, $\Delta ^b_k$ also satisfy
\begin{equation}\begin{split}\label{eq8.9} \text{a)}\quad &E^b_kE^b_j=E^b_jE^b_k=E^b_k,\quad j\geq k\\
\text{b)} \quad & \Delta^b_j\Delta^b_k=0\quad j\neq k\\
\text{c)}\quad & (\Delta^b_k)^2=\Delta_k\\
\text{d)}\quad & \lim_{k\to \infty}E^b_k f=f\text{ a.e.},\quad f\in L^p_{\loc}\,,\,p\geq 1\\
\text{e)}\quad & \lim_{k\to \infty}E^b_kf=0,\quad f\in L^p,\;\; 1\leq p<\infty\\
\text{f)}\quad & \sum(\Delta^b_k)^2=\sum \Delta^b_k=I.\end{split}\end{equation}

We shall also find it useful to consider the transposes of the operators $E^b_k$, $\Delta^b_k$, which we denote as follows:
\begin{equation*} A^b_k\equiv (E^b_k)^{\tr} =b\frac{E_k}{E_k(b)},\qquad
D^b_k = A^b_{k+1}-A^b_k=(\Delta ^b_k)^{\tr}.\end{equation*}
One may readily verify that for $b\in D\psi A$ the operators $A^b_k$, $D^b_k$ satisfy the properties enjoyed by $E^b_k$,
$\Delta^b_k$ in \eqref{eq8.9}. Moreover, we have
\begin{proposition}\label{p8.10} If $b\in D\psi A$ then
$$\sum_k\|D^b_kf\|^2_2\leq C\| f\|^2_2.$$\end{proposition}
\begin{proof} Observe that $A^b_kf=bE_kf/ E_kb$. Hence
\begin{equation*}|D^b_kf|\leq |b|\left| \frac{E_{k+1}f}{E_{k+1}b}-\frac{E_kf}{E_kb}\right|
\leq \| b\|_\infty \left(\frac{|\Delta_kf|}{|E_{k+1}b|} +\frac{|E_kf||\Delta_k b|}{|E_{k+1}b||E_kb|}\right)
.\end{equation*}
The conclusion of the proposition now follows from \eqref{eq8.2}, Proposition~\ref{p8.6}, dyadic pseudo-accretivity,
and the dyadic version of Carleson's Lemma. We omit the details.\end{proof}

Next, we introduce some further terminology.

\begin{definition}\label{d8.11} Given a dyadic cube $Q\subseteq \mathbb{R}^n$, a ``discrete Carleson region" is the
collection
$$R_Q\equiv \{\text{dyadic }Q'\text{ such that } Q'\subseteq Q\}.$$
We shall refer to $Q$ as the ``top" of $R_Q$.
We remark that in using the term ``discrete Carleson region" in this fashion, we are implicitly identifying a cube $Q'$ with its associated
``Whitney box" $Q' \times \left[\ell(Q')/2,\ell(Q')\right]$.\end{definition}

\begin{definition}\label{d8.12} Given a dyadic cube $Q\subseteq \mathbb{R}^n$, a ``discrete sawtooth region" is the
collection
$$\Omega \equiv R_Q\backslash (\cup R_{P_j}),$$
where $\{ P_j\}$ is a family of non-overlapping dyadic sub-cubes of $Q$.\end{definition}

\begin{definition}\label{d8.13} We say that $b$ is ``$q$-dyadically pseudo accretive on a sawtooth domain $\Omega$"
($b\in q-D\psi A (\Omega)$), if there exist constants $\delta >0$ and $C_0<\infty$ such that for every $Q'\in
\Omega$
\begin{enumerate}\item[(i)] \quad$\left|\frac{1}{|Q'|} \int_{Q'} b\right|\geq \delta$
\item[(ii)]\quad $\frac{1}{|Q'|} \int_{Q'} |b|^q\leq C_0$.\end{enumerate}\end{definition}

We now introduce some alternative notation, which we shall find useful when working with discrete sawtooth regions.
For $Q\in \mathbb{D}_k$, we set
$$D_Q^b f(x)\equiv 1_Q(x)D^b_kf(x)$$
and we adapt the analogous convention for $A^b_k(A^b_Q)$, $\Delta^b_k(\Delta^b_Q)$ and $E^b_k(E^b_Q)$. Since the cubes
in a given dyadic scale are non-overlapping, we have, for example
$$\sum_Q \| D^b_Qf\|^2_2=\sum^\infty_{k=-\infty}\| D_k^bf\|^2_2,$$
where the first sum runs over all dyadic cubes.

We also describe a convenient splitting of a discrete sawtooth region as follows. Given a dyadic cube $Q_1$, and a
discrete sawtooth
$$\Omega \equiv R_{Q_1}\backslash (\cup R_{P_j}),$$
we split
$$\Omega \equiv \Omega _1 \cup \Omega_{\text{buffer}},$$
where
$$\Omega_{\text{buffer}} \equiv \{ Q\in \Omega:Q\text{ has at least one child not in }\Omega \}.$$
Thus, if $Q\in \Omega_1$, then every child of $Q$ belongs to $\Omega$. We have the following extension of
Proposition~\ref{p8.10}:

\begin{lemma}\label{l8.14} Let $\Omega \equiv R_{Q_1}\backslash(\cup R_{P_j})$ be a discrete sawtooth region
corresponding to a dyadic cube $Q_1$, and let $\Omega_1\cup\Omega_{\text{buffer}}$ be the splitting of $\Omega$
described above. Suppose also that $b\in 2-D\psi A(\Omega)$. Then
$$\sum_{Q\in \Omega_1}\| D^b_Qf\|^2_2\leq C \| f\|^2_{L^2(Q_1)}.$$\end{lemma}

\begin{proof} Fix $Q\in \mathbb{D}_k\cap \Omega_1$. By definition,
\begin{multline*}\| D^b_Q f\|^2_2= \int_Q |D^b_kf|^2
=\int_Q\left| b\left( \frac{E_{k+1}f}{E_{k+1}b}-\frac{E_kf}{E_kb}\right)\right|^2\\
=\sum_{\substack{Q'\in \mathbb{D}_{k+1}\\ Q'\subseteq Q}}\int_{Q'} 
\left|b\left(\frac{E_{Q'}f}{E_{Q'}b}-\frac{E_Qf}{E_Qb}\right)\right|^2
=\sum_{\substack{Q'\in \mathbb{D}_{k+1}\\Q'\subseteq Q}} \left| \frac{E_{Q'}f}{E_{Q'}b}-\frac{E_Qf}{E_Q b}\right|^2 \int_{Q'} |b|^2\\
= \sum_{\substack{Q'\in \mathbb{D}_{k+1}\\ Q'\subseteq Q}} \int_{Q'} \left| \frac{\Delta _Q f}{E_{Q'}b}-\frac{E_Q f\Delta _Q b}{E_{Q' } b\, E_Q b}\right|^2 \frac{1}{|Q'|}\int_{Q'} |b|^2 ,\end{multline*}
where in the last two steps we have used that $E_{Q'}$, $E_Q$ are constant on $Q'$. But if $Q\in \Omega_1$, then its children $Q'$ all belong to $\Omega$. Since $b\in 2-D\psi A(\Omega)$, the last expression is therefore bounded by
\begin{equation*} C\int_Q \left(|\Delta_Qf|^2+ |E_Qf|^2 |\Delta_Qb|^2\right).\end{equation*}
Summing over $Q\in \Omega_1$ yields the desired estimate, once we have proved the following analogue of the discrete Fefferman-Stein Carleson measure estimate Proposition~\ref{p8.6}.
\end{proof}

\begin{lemma}\label{l8.15} Let $Q_1$, $\Omega=\Omega_1\cup \Omega_{ \buffer}$ be as in the previous Lemma, and suppose that $b\in 2-D\psi A(\Omega)$. Then 
\begin{equation*} \sup \frac{1}{|\widetilde{Q}|} \sum_{Q\in \Omega_1,Q\subseteq \widetilde{Q}} \| \Delta_Q b\|^2_2\leq C C_0,\end{equation*}
where $C_0$ is the constant in Definition~\ref{d8.13}, and where the supreme runs over all dyadic $\widetilde{Q}\subseteq Q_1$.\end{lemma}

\begin{proof} We observe that
\begin{equation*} \sum_{\substack{Q\in \Omega_1\\ Q\subseteq 
\widetilde{Q}\subseteq Q_1}} \| \Delta _Q b\|^2_2 =
\sum_{\substack{Q\in \Omega_1\\Q\subseteq \widetilde{Q} 
\subseteq Q_1}} \| \Delta _Q (1_{\widetilde{Q}} \,b)\|^2_2\end{equation*}
is non-zero only if $\widetilde{Q}\in \Omega$. But $b\in 2-D\psi A(\Omega)$, so by \eqref{eq8.2} we have that
\begin{equation*}\sum_{Q\,\,\text{dyadic}} \| \Delta _Q (1_{\widetilde{Q}} \,b)\|^2_2 \leq 
C\int_{\widetilde{Q}} |b|^2\leq CC_0 |\widetilde{Q}|.\end{equation*}
This concludes the proof of Lemma~\ref{l8.15} and hence also that of Lemma~\ref{l8.14}.\end{proof}

\section{Proof of Theorem~\ref{t6.6} (Local $Tb$ Theorem for singular integrals)
\label{s8}}

We now proceed to give the proof of Theorem~\ref{t6.6}. The proof follows that of Theorem 6.8 in \cite{AHMTT}, which for the sake of expository simplicity treated only the case of 
``perfect dyadic" Calder\'on-Zygmund kernels in one dimension. The more general version given here, in which the ``perfect dyadic" cancellation condition is replaced by \eqref{eq6.3}(b), will entail dealing with a moderate amount of purely technical complication, but the gist of the proof is unchanged.

By the T1 theorem, plus a localization argument, 
it is enough to show that there is a constant $C$, depending only on dimension, the kernel bounds in \eqref{eq6.3}, and the constants in hypotheses (i), (ii) and (iii) of the 
Theorem, such that for every dyadic cube $Q$,
\begin{align*}\label{T1}\tag{T$1_{\loc}$}\text{(a)}&\quad \| T1_Q \|_{L^1(Q)} \leq C|Q|\\
\text{(b)}&\quad \| T^{\tr} 1_Q\|_{L^1(Q)} \leq C|Q|\notag\end{align*}
Indeed, it is well known that one may deduce both the weak boundedness property, 
and that $T1$, $T^{\tr}1\in BMO,$ 
from ($T1_{\loc}$),   \eqref{eq6.3} and \eqref{eq6.13}. We omit the details. 
In the sequel we shall use the generic $C$ to denote a constant depending only on the benign parameters listed above.

Now, by the symmetry of our hypotheses, it will suffice to establish only \eqref{T1}(b), and we do this for $Q$ contained in same fixed cube $Q_{  \bg}$. Since $Q_{  \bg}$ is arbitrary, the general case follows, as long as our constants are independent of $Q_{  \bg}$ (as they will be).

We thus fix $Q_{  \bg}$, and define
$$B_1\equiv \sup \frac{1}{|Q|} \| T^{\tr} 1_Q\|_{L^1(Q)},$$
where the supremum runs over all dyadic $Q\subseteq Q_{  \bg}$. By our qualitative hypothesis that $K\in L^\infty$, we see that $B_1<\infty$, although apparently it may depend on $\| K\|_\infty$ and $Q_{  \bg}$.  However, we shall show that there exists $\epsilon >0$, depending only on the allowable parameters, such that for every $Q\subseteq Q_{\bg}$, and for every $f\in L^\infty (Q)$ with $\| f\|_\infty \leq 1$, we have the estimate
\begin{equation}\label{eq8.16} |\int_QTf\,|\leq (1-\epsilon )B_1 |Q|+C|Q|.\end{equation}
By duality, this proves that $B_1\leq (1-\epsilon )B_1+C$, and \eqref{T1}(b) follows.

In the sequel, we shall use the following convenient notational convention:
\begin{equation*} \frac{1}{|Q|} \int_Q f=[f]_Q.\end{equation*}
By renormalizing, we may assume that hypothesis (iii) of the Theorem reads
\begin{equation}\label{eq8.17} [b^1_Q]_Q=1=[b^2_Q]_Q.\end{equation}

\begin{lemma}\label{l8.18} Suppose that $\{ b_Q\}$ satisfies (as in the hypotheses of Theorem~\ref{t6.6})
\begin{enumerate}\item[(i)] \quad$\int_Q|b_Q|^q\leq C|Q|$, for some $q>2$
\item[(ii)] \quad$\int_Q |Tb_Q|^2\leq C|Q|$
\item[(iii)] \quad$[b_Q]_Q=1$,\end{enumerate}
and that $\supp b_Q\subseteq Q$. Then there exists $\epsilon >0$, and for each fixed $Q_1$ a partition of $R_{Q_1}$ into
$$R_{Q_1}=\Omega _1 \cup \Omega_{ \buffer} \cup (\cup R_{P_j}),$$
where the tops $\{ P_j\}$ are non-overlapping dyadic sub-cubes of $Q_1$, such that if $b\equiv b_{Q_1}$, then

\begin{align}\label{eq8.19} & \quad \sum |P_j|\leq (1-\epsilon )|Q_1|\\
&\label{eq8.20}b\in q-D\psi A(\Omega_1 \cup \Omega_{ \buffer} )\\
&\label{eq8.21} \sup_{Q\subseteq \tilde{Q}\subseteq 2Q} [(Mb)^2]_{\tilde{Q}} \leq C,\end{align}
for all $Q\in \Omega_1\cup \Omega_{ \buffer}$ (here, $2Q$ denotes the concentric double of $Q$);
\begin{align}\label{eq8.22} &  [|Tb|^2]_Q \leq C,\quad \forall Q\in \Omega_1\cup \Omega_{ \buffer}\\
\label{eq8.23} & \qquad \sum_{Q\in \Omega_{ \buffer}} |Q|\leq C|Q_1|
\end{align} \begin{equation}
\label{eq8.24}  f=[f]_{Q_1}b+\sum_{Q\in \Omega_1}D^b_Qf+\sum_j (f1_{P_j}-[f]_{P_j} b_{P_j}) +\sum_{Q\in \Omega_{ \buffer}}\zeta_Q,\end{equation}
where
\begin{equation*}\zeta_Q\equiv S^b_Q f+\sum_{P_j \,\,\text{children of }Q} [f]_{P_j}b_{P_j},\end{equation*}
and, for $x\in Q'$, and $Q'$ a child of $Q\in \Omega_{ \buffer}$, 
\begin{equation*}S^b_Qf(x)\equiv \begin{cases} D^b_Q f(x), & x\in Q'\in (\Omega_1\cup \Omega_{\buffer})\\
-A^b_Q f(x), & x\in Q'\notin (\Omega_1 \cup \Omega_{ \buffer})\end{cases}.\end{equation*}
Furthermore $\int  \zeta_Q=0$, and $\|  \zeta_Q\|_2 \leq C|Q|^{1/2}$.\end{lemma}

\begin{proof}[Proof of the lemma] We begin by verifying the claimed
properties of $\zeta_Q$, for $Q\in \Omega_{\buffer}$,
assuming \eqref{eq8.20}. By definition of 
$S^b_Q$,
\begin{equation*} \zeta_Q=\sum_{\substack{Q'\in\Omega\\ Q'\text{child of } Q}} \frac{b}{[b]_{Q'}} [f]_{Q'}1_{Q'}+ \sum_{\substack{Q'\notin\Omega\\ Q '\text{child of } Q}} [f]_{Q'}b_{Q'}- \frac{b}{[b]_Q} [f]_Q 1_Q ,\end{equation*}
where in the middle term we have used that  $[b_{Q'}]_{Q'}=1$, and that if $Q'$ is a child of $Q \in \Omega_{\buffer}$, with $Q'\notin \Omega \equiv \Omega_1 \cup \Omega_{\buffer}$, then $Q'=P_j$ for some $j$. It is now routine to verify that $\int\zeta_Q=0$, since $[b_{Q'}]_{Q'}=1$. Clearly, $\supp \zeta_Q\subseteq Q$. Also, the bound 
\begin{equation*} \| \zeta_Q\|_2\leq C\| f\|_\infty |Q|^{\frac{1}{2}} \leq C|Q|^{\frac{1}{2}}\end{equation*} follows from \eqref{eq8.20} and H\"older's inequality.

We now turn to the main part of the proof. By hypothesis (i) of the Lemma, applied to $b$ in $Q_1$, and by the $L^q$ boundedness of the maximal function, we have that
\begin{equation}\label{eq8.25} \int_{\mathbb{R}^n}(Mb)^q\leq C\int_{Q_1}|b|^q\leq C|Q_1|\end{equation}
where we have used that $b$ is supported in $Q_1$. We now perform a standard stopping time argument, subdividing $Q_1$ dyadically to extract a collection of sub-cubes $\{ P_j\}$ which are maximal with respect to the property that for some $\delta >0$ to be chosen, at least one of the following holds:
\begin{equation}\begin{split}\label{eq8.26} \text{(1)} & \quad |[b]_{P_j}|\leq \delta\\
\text{(2)}& \quad \sup_{\tilde{Q}:P_j\subseteq \tilde{Q}\subseteq 2P_j} [(Mb)^q]_{\tilde{Q}} +[|Tb|^2]_{P_j} \geq \frac{C}{\delta^2}\end{split}\end{equation}
As usual, we then set $\Omega \equiv R_{Q_1}\backslash (\cup R_{P_j})$, and we further decompose $\Omega=\Omega_1 \cup \Omega_{\buffer}$, where as above 
\begin{equation*} \Omega_{\buffer} =\{ Q\in \Omega :Q\text{ has at least one child not in } \Omega\}.\end{equation*}
Then \eqref{eq8.20}, \eqref{eq8.21} and \eqref{eq8.22} hold by construction. The representation \eqref{eq8.24} holds by definition of $D^b_Q$ and $S^b_Q$, by the normalization $[b_Q]_{Q}=1$, and by the telescoping nature of sums involving the $D^b_k$ operator. Furthermore, since each $Q\in \Omega_{\buffer}$ contains at least one bad child $P_j$,
we have that
\begin{equation*} \sum_{Q\in \Omega_{\buffer}}|Q|\leq 2^n\sum |P_j|\leq 2^n|Q_1|,\end{equation*}
which is \eqref{eq8.23}. It therefore remains only to verify \eqref{eq8.19}. To this end, we assign each ``bad" cube $P_j$ to a family $S_1$ or $S_2$, according to whether $P_j$ satisfies property (1) or (2) of \eqref{eq8.26}. If it happens to satisfy both of these inequalities, then we assign it arbitrarily to $S_1$. We then define
\begin{align*} \bad_1&=\cup_{P_j\in S_1} P_j,\quad \bad_2=\cup_{P_j\in S_2}P_j\\
\intertext{and}
\good &=Q_1\backslash (\bad_1\cup \bad_2).\end{align*}
Then by hypothesis (iii) of the lemma,
\begin{equation*}\begin{split}|Q_1|&=\int_{Q_1}b=\int_{\good} b+\int_{\bad_1}b+\int_{\bad_2}b\\
&\leq |\good |^{\frac{1}{2}} \| b\|_{L^2(Q_1)} +\delta \sum |P_j|+|\bad _2|^{\frac{1}{2}} \| b\|_{L^2(Q_1)},\end{split}\end{equation*}
where we have used \eqref{eq8.26}(1) to control the middle term. Now, by hypothesis (i) of the Lemma and H\"older's inequality, we have that $\| b\|_2 \leq C|Q_1|^{\frac{1}{2}}$, whence 
\begin{equation}\label{eq8.27} (1-\delta) |Q_1|\leq C|\good |^{\frac{1}{2}}|Q_1|^{\frac{1}{2}} +|\bad _2|^{\frac{1}{2}} |Q_1|^{\frac{1}{2}}.\end{equation}
Choosing $\delta>0$ sufficiently small, we will obtain the conclusion of the Lemma once we show that
$$|\bad_2|\leq C\delta^2|Q_1|.$$
To this end, we observe that by \eqref{eq8.26}(2) and the Hardy-Littlewood Theorem,
\begin{equation*}\begin{split} |\bad_2 |&\leq \left|\left\{ M(Mb)^q)>\frac{C}{2\delta^2}\right\}\right|+\left|\left\{ M(|Tb|^21_{Q_1})>\frac{C}{2\delta ^2}\right\}\right|\\
&\leq  C\delta^2\left( \int_{\mathbb{R}^n}(Mb)^q+\int_{Q_1}|Tb|^2\right)\,\leq\, C\delta^2|Q_1|,\end{split}\end{equation*}
as desired. This concludes the proof of Lemma~\ref{l8.18}.
\end{proof}

We now return to the proof of \eqref{eq8.16}. Fix a cube $Q_1$, and let $f$ be supported in $Q_1$, with $\| f\|_\infty \leq 1$. We apply Lemma~\ref{l8.18} in the cube $Q_1$, with $b_Q=b^1_Q$, $b=b^1_{Q_1}\equiv b_1$, so that we have a decomposition $R_{Q_1}=\Omega_1\cup \Omega_{\buffer} \cup (\cup P_j)$, for which \eqref{eq8.19}-\eqref{eq8.23} are satisfied, and furthermore $f$ may be decomposed as in \eqref{eq8.24}. We need to estimate $\left|\int_{Q_1} Tf\right|$, so by \eqref{eq8.24} it is enough to consider
\begin{equation*}\begin{split} |[f]_{Q_1} | \int_{Q_1} |Tb_1|&+\left|\sum_{Q\in \Omega_1}TD_Q^{b_1}f\right|
+\left| \sum_j\int_{Q_1}T(f1_{P_j}-[f]_{P_j}b^1_{P_j})\right|\\
&+ \left| \sum_{Q\in \Omega_{\buffer} }\int_{Q_1} T\zeta_Q\right|\, \equiv \,|\I|+|\II|+|\III |+|\IV |.\end{split}\end{equation*}
By hypothesis (ii) of Theorem~\ref{t6.6} and Cauchy-Schwarz, we have that
$$|\I|\leq C\| f\|_\infty |Q_1|\leq C|Q_1|.$$
Term II is the main term, and we defer its treatment momentarily. Next, we consider term III. For notational convenience, we set 
$$f_j\equiv f1_{P_j} -[f]_{P_j} b^1_{P_j}.$$
Since $[b^1_{P_j}]_{P_j}=1$, we have that $\int f_j=0$. Moreover, $\supp f_j\subseteq P_j$, and
\begin{equation}\label{eq8.28} \| f_j\|_2\leq C\| f\|_\infty |P_j|^{1/2}.\end{equation}
We now claim that
\begin{equation}\label{eq8.29}\III =\sum_j \int_{P_j} Tf_j+0 (\| f\|_\infty |Q_1|).\end{equation}
Indeed,
\begin{equation}\label{eq8.30} \int_{Q_1\backslash P_j} Tf_j=\int_{Q_1\backslash 2P_j} Tf_j+\int_{(Q_1\cap 2P_j)\backslash P_j}Tf_j.\end{equation}
The second term is dominated in absolute value by
\begin{equation*} C|P_j|^{\frac{1}{2}}\left( \int_{2P_j\backslash P_j}|Tf_j|^2\right)^{\frac{1}{2}} 
\leq C|P_j|^{\frac{1}{2}}\| f_j\|_2
\leq C\| f\|_\infty |P_j|,\end{equation*}
where the first inequality is essentially dual to \eqref{eq6.13}, by the kernel condition \eqref{eq6.3}(a) and the fact that $\supp f_j\subseteq P_j$, and the second inequality is just \eqref{eq8.28}. The first term in \eqref{eq8.30} may be handled by the classical Calder\'on-Zygmund estimate, using \eqref{eq6.3}(b) and the fact that $\int f_j=0$, and we obtain the bound
\begin{equation*}\begin{split} &C\iint_{|x-y|>C\ell (P_j)} \frac{\ell (P_j)^\alpha}{|x-y|^{n+\alpha}} |f_j(y)|dxdy\\
&\quad \leq C\| f_j\|_1\leq C|P_j|^{\frac{1}{2}} \| f_j\|_2\leq C\| f\|_\infty |P_j|.\end{split}\end{equation*}
Summing in $j$, we obtain \eqref{eq8.29}.

Thus, to finish our treatment of term III, we need only observe that
\begin{multline*}\left| \sum_j\int_{P_j}Tf_j\right| \leq \left| \sum_j \int_{P_j} T(f1_{P_j})\right| + \left| \sum_j \left( \int_{P_j}Tb^1_{P_j}\right) [f]_{P_j}\right|\\
\leq B_1 \| f\|_\infty \sum_j |P_j|+C\| f\|_\infty \sum |P_j|,\end{multline*}
where we have used the definition of $B_1$ and hypothesis (ii) of Theorem~\ref{t6.6}. From \eqref{eq8.19} and the normalization $\| f\|_\infty \leq 1$, we obtain the bound
$$|\III |\leq B_1 (1-\epsilon )|Q_1|+C|Q_1|.$$

We now consider term IV. By Lemma~\ref{l8.18} and the definition of $\zeta_Q$, we have that
$$\supp \zeta_Q \subseteq Q,\quad \int \zeta_Q=0, \,\,\text{and}\,\,\| \zeta_Q\|_2 \leq C|Q|^{1/2}.$$
Thus, from the same argument used to establish \eqref{eq8.29}, we obtain 
\begin{equation}\label{eq8.31} \IV =\sum_{Q\in \Omega_{\buffer}} \int_Q T\zeta_Q +O (|Q_1|),\end{equation}
where in the ``big $O$" term we have used \eqref{eq8.23}. We recall that $$\zeta_Q =S_Q^{b_1} f+\sum_{P_j\text{ children of } Q} [f]_{P_j} b^1_{P_j},$$ where 
for $x\in Q'$, with $Q'$ a child of $Q\in \Omega_{\buffer}$, we have either that
$$S^{b_1}_Q f(x)=-b_1 (x)\frac{\int_Qf}{\int_Q b_1},$$
if $Q'\notin \Omega \equiv (\Omega_1 \cup \Omega_{\buffer})$ (in which case we say that $Q'$ is a ``bad" child of $Q$) or 
$$S^{b_1}_Q f(x)=b_1 (x)\left[\frac{\int_{Q'}f}{\int_{Q'} b_1}-\frac{\int_Qf}{\int_Q b_1}\right],$$
if $Q'\in \Omega$ ($Q'$ is a ``good" child of $Q$).

Now, by \eqref{eq8.20}, $b_1 \in q-D\psi A (\Omega)$ (Definition~\ref{d8.13}), so that
\begin{multline}\label{eq8.32} \left| \int_Q T\zeta_Q\right| \leq \frac{C}{\delta} \left( \sum_{Q'\text{ good child of } Q} \left| \int_Q T(b_1 1_{Q'}) \right| + \left| \int_Q T(b_1 1_Q)\right| \right) \\
+ \sum_{Q'\text{ bad child of } Q} \left| \int_Q Tb^1_{Q'}\right|,\end{multline}
where in the last term we have used that the bad children of $Q$ are precisely those $P_j$ which are children of $Q$.

We shall estimate this last expression via the following

\begin{lemma}\label{l8.33} Suppose that $Q\subseteq Q_1$. Then with
$b_1 \equiv b^1_{Q_1}$, we have
$$\int_{3Q} |T(b_11_Q)|^2\leq C\int_{Q}|Tb_1|^2+\int_{2Q}|b_1|^2+\int_{Q}(M(b_1))^2$$ and similarly for $b^2_{Q_1}$, $T^{\tr}$.\end{lemma}

Let us take the lemma for granted momentarily. In \eqref{eq8.32}, $Q'$ is a child of $Q$, hence the concentric triple $3Q'$ contains $Q$. Moreover, the ``good" children, being in $\Omega$, satisfy \eqref{eq8.21} and \eqref{eq8.22}, with $b=b_1$. Consequently, we may apply the lemma to $Q'\subseteq Q_1$ or to $Q\subseteq Q_1$ in the first two terms on the right side of \eqref{eq8.32} to obtain the bound
\begin{equation*}\frac{C}{\delta} \left( \int_Q |Tb_1|^2+\int_{2Q}|b_1|^2+\int_Q(M(b_1))^2\right) \leq \frac{C}{\delta} |Q|.\end{equation*}
In addition, the last term in \eqref{eq8.32} is no larger then
$$\sum_{Q'}\left( \left| \int_{Q\backslash Q'} Tb^1_{Q'}\right| +\left|\int_{Q'}Tb^1_{Q'}\right|\right) \leq C\sum_{Q'} |Q'|\leq C|Q|,$$
by the dual estimate to \eqref{eq6.13}, plus hypotheses (i) and (ii) of Theorem~\ref{t6.6}. Since $\delta >0$ is fixed, summing over $Q$ in $\Omega_{\buffer}$ yields that
$$|\IV\|\leq C|Q_1|,$$
by \eqref{eq8.23}.

Combining our estimates for I, III and IV, we have therefore proved that
\begin{equation}\label{eq8.34}\left| \int_{Q_1} Tf\right| \leq |\II |+C|Q_1| +B_1 (1-\epsilon) |Q_1|,\end{equation}
modulo the proof of Lemma~\ref{l8.33}, which we shall give now, before embarking on our treatment of the math term II.

\begin{proof}[Proof of Lemma~\ref{l8.33}] The proof is based on another Lemma.

\begin{lemma}\label{l8.35} For all dyadic $Q$, and for every $f\in L^2(Q)$, we have that
$$\| f\|_{L^2(Q)} \leq C\left(\| f -[f]_Q \|_{L^2(Q)} +|Q|^{-\frac{1}{2}} |\langle f,b^2_Q\rangle |\right),$$
and similarly for $b^1_Q$.\end{lemma}

We first show that this lemma yields Lemma~\ref{l8.33}. By the dual estimate to \eqref{eq6.13}, we have that
$$\int_{3Q\backslash Q} |T(b_1 1_Q)|^2\leq C\int_Q|b_1|^2.$$
Thus, it suffices to show that $\int_{Q} |T(b_1 1_Q)|^2\leq \beta$, where
$$\beta \equiv \int_Q |Tb_1 |^2+\int_{2Q} |b_1 |^2+\int_Q M(b_1)^2.$$
We note that
\begin{equation*}|\langle T(b_1 1_Q),b^2_Q\rangle |= |\langle b_1 1_Q,T^{\tr} b^2_Q \rangle |
\leq \| b_1 \| _{L^2(Q)} \, \| T^{\tr} b^2_Q\| _{L^2(Q)}
\,\leq\, C|Q|^{1/2 }\,\| b_1\| _{L^2(Q)},\end{equation*}
by hypothesis (ii) of Theorem~\ref{t6.6}. Thus, by Lemma~\ref{l8.35}, with $f=T(b_1 1_Q)$, it suffices to show that
\begin{equation}\label{eq8.36} \| f-[f]_Q\|_{L^2(Q)} \leq C\sqrt{\beta}.\end{equation}
In turn, \eqref{eq8.36} will follow if we can show that, for all $h\in L^2(Q)$ with $\int_Q h=0$, we have
$$|\langle f,h\rangle | \leq C\| h\|_2 \sqrt{\beta}.$$
But
\begin{equation*} \langle f,h\rangle =\langle b_1 1_Q, T^{\tr }h\rangle 
= \langle b_1, T^{\tr} h\rangle -\langle b_11_{(2Q\backslash Q)}, 
T^{\tr} h\rangle -\langle b_1 1_{(2Q)c} T^{\tr} h\rangle
\equiv U+V+W.\end{equation*}
Now
$$|U|\leq \| Tb_1 \| _{L^2(Q)} \| h\|_{L^2(Q)} \leq C\sqrt{\beta} \| h\|_2.$$
Moreover, we have that
$$|V|\leq \| b_1 \|_{L^2(2Q)} \| T^{\tr }h\|_{L^2(2Q\backslash Q)} \leq C\sqrt{\beta} \| h\|_2,$$
where we have used the dual estimate to \eqref{eq6.13} in the last step. Finally, since $\int h=0$, we have by the standard Calder\'on-Zygmund estimate that
\begin{equation*} |W|\leq 
\int_{(2Q)^c} |b_1(y)|\int_Q|h(x)|\frac{\ell (Q)^\alpha}{|x-y|^{n+\alpha}} dxdy
\leq \int_Q |h(x)|\, M(b_1)(x)\, dx\leq C\| h\|_2 \sqrt{\beta}.\end{equation*}
Thus, Lemma~\ref{l8.35} implies Lemma~\ref{l8.33}. \end{proof}

We now give the
\begin{proof}[Proof of Lemma~\ref{l8.35}] Let $h\in L^2(Q)$, with $\| h\|_2=1$. Then
\begin{equation*}\langle f,h\rangle =\langle f,h-[h]_Q b^2_Q\rangle +[h]_Q\langle f,b^2_Q\rangle
= \langle f-[f]_Q , h-[h]_Q b^2_Q\rangle +[h]_Q \langle f,b^2_Q\rangle ,\end{equation*}
where we have used that $\int_Q(h-[h]_Qb^2_Q)=0$, since $[b^2_Q]_Q = 1$. Thus, by Cauchy-Schwarz,
$$|\langle f,h\rangle |\leq \| f-[f]_Q\|_{L^2(Q)} \left(1+|[h]_Q|\,\| b^2_Q\|_2\right)+|[h]_Q|\, |\langle f,b^2_Q\rangle |.$$
But $$|[h]_Q|\leq \left(\frac{1}{|Q|} \int_Q |h|^2\right)^{1/2} \leq |Q|^{-\frac{1}{2}},$$ and by hypothesis (i), $\| b^2_Q\|_2 \leq C|Q|^{1/2}$. The conclusion of the lemma now follows readily.\end{proof}

Next, we return to \eqref{eq8.34}, and more precisely, to the term
$$\II =\sum_{Q\in \Omega_1} \int_{Q_1} TD^{b_1}_Q f,$$
where $f$ is supported in $Q_1$, and $\| f\|_\infty \leq 1$. Having established \eqref{eq8.34}, we must now show that $|\II |\leq C|Q_1|$, whence \eqref{eq8.16} follows, since $Q_1$ is arbitrary. But
$$\II =\sum_{Q\in \Omega_1} \langle \Delta_Q^{b_1} T^{\tr} 1_{Q_1} ,D^{b_1}_Q f\rangle,$$
because $(D^{b_1}_Q)^2=D^{b_1}_Q$, and $(D^{b_1}_Q)^{\tr}=\Delta^{b_1}_Q$.
Thus
\begin{equation}\label{eq8.37} | \II |\leq \left( \sum_{Q\in \Omega_1} \| D^{b_1}_Q f\|^2_{L^2(Q)} \right)^{\frac{1}{2}} \left( \sum_{Q\in \Omega_1} \| \Delta^{b_1}_Q T^{\tr} 1_{Q_1}\|^2_{L^2(Q)}\right) ^{\frac{1}{2}}.\end{equation}
Since $b_1$ satisfies \eqref{eq8.20}, we have by Lemma~\ref{l8.14} that the first factor on the right side of \eqref{eq8.37} is bounded by $C\| f\|_{L^2(Q_1)}\leq C|Q_1|^{1/2}$. It is therefore enough to show that the second factor is also dominated by $C|Q_1|^{1/2}$. More generally, setting
$$B_2\equiv \sup_{Q_2\subseteq Q_1} \frac{1}{|Q_2|} \sum_{Q\in \Omega_1\cap R_{Q_2}} \| \Delta^{b_1}_QT^{\tr}1_{Q_1}\|^2_{L^2(Q)},$$
we shall show that $B_2\leq C$. More precisely, for $Q_2 \subseteq Q_1$ now fixed, we shall show that
\begin{equation}\label{eq8.38} \sum_{Q\in \Omega_1\cap R_{Q_2}} \| \Delta^{b_1}_Q T^{\tr} 1_{Q_1} \|^2_{L^2(Q)} \leq (1-\epsilon )B_2 |Q_2|+C|Q_2|.\end{equation}
Once \eqref{eq8.38} is established, we shall be done. To this end, we decompose $R_{Q_2}$ as in Lemma~\ref{l8.18}, with respect to $b=b^2_{Q_2} \equiv b_2$. In particular, $R_{Q_2}=\Omega_2 \cup \Omega_{2,\buffer} \cup (\cup R_{P^2_i})$, where
$$ \sum |P^2_i|\leq(1-\epsilon )|Q_2|,\quad \sum_{Q\in \Omega_{2,\buffer}} |Q|\leq C|Q_2|,$$
and $b_2\in q-D\psi A$ on $\Omega_2 \cup \Omega_{2,\buffer}$. The left hand side of \eqref{eq8.38} then splits into
$$\sum_{Q\in \Omega_1\cap \Omega_2} +\sum_{Q\in \Omega_1\cap \Omega_{2,\buffer}} +\sum_i \sum_{Q\in \Omega_1 \cap R_{P^2_i}} \equiv \Sigma_1+\Sigma_2+\Sigma_3.$$
Now, by definition of $B_2$,
\begin{equation*} \Sigma_3\equiv \sum_i \sum_{Q\in 
\Omega_1\cap R_{P^2_i}} \| \Delta_Q^{b_1} T^{\tr} 1_{Q_1} \|^2_{L^2(P^2_i)}
\leq B_2 \sum_i |P^2_i|\leq B_2 (1-\epsilon )|Q_2|.\end{equation*}
Next, we consider $\Sigma_2$. For $Q\in \Omega_1 \cap \Omega_{2,\buffer}$, we write $1_{Q_1}=1_{Q_1\backslash 2Q} +1_{(Q_1\cap 2Q)\backslash Q} +1_Q.$ Since 
$b_1\in q-D\psi A (\Omega_1 \cup \Omega _{\buffer})$, we have that $\Delta^{b_1}_Q :L^2(Q)\to L^2(Q)$. Thus, using also \eqref{eq6.13}, we obtain
$$\| \Delta_Q ^{b_1} T^{\tr} 1_{(Q_1\cap 2Q)\backslash Q} \| ^2_2 \leq C\| 1_{2Q\backslash Q} \|^2_2 \leq C|Q|.$$
Summing this term over $Q\in \Omega_{2,\buffer}$ yields the bound $C|Q_2|$ as desired. Also $\Delta_Q^{b_1}1=0$. Thus, if we denote by $\varphi^{b_1}_Q(x,y)$ the kernel of $\Delta ^{b_1}_Q$, we have by \eqref{eq6.3}(b) that
\begin{equation*} |\Delta^{b_1}_Q T^{\tr} 1_{Q_1 \backslash 2Q}(x)|\leq C\int |\varphi^{b_1}_Q (x,y)|\,dy\int_{|z-y_Q|>c\ell (Q)}\frac{\ell(Q)^\alpha}{|z-y_Q|^{n+\alpha}}dz
\leq C,\end{equation*}
where $y_Q $ is the center of $Q$. Therefore
$$\| \Delta^{b_1}_QT^{\tr} 1_{Q_1\backslash 2Q} \|^2_{L^2(Q)} \leq C|Q|,$$
and we can again sum over $Q\in \Omega_{2,\buffer}$ to obtain the bound $C|Q_2|$.

To finish our treatment of $\Sigma_2$, it remains to consider the contribution of $1_Q$. By definition,
\begin{equation*}\begin{split} \varphi_Q^{b_1}(x,y)&=-1_Q(x)1_Q(y)\frac{1}{|Q| }\frac{b_1(y)}{[b_1]_Q}+\sum_{Q'\text{ children of } Q} 1_{Q'} (x)1_{Q'} (y) \frac{1}{|Q'|} \frac{b_1 (y)}{[b_1]_{Q'}}\\
&\equiv \lambda^{b_1}_Q (x,y)\,b_1(y).\end{split}\end{equation*}
Then,
\begin{equation*} \Delta^{b_1}_Q T^{\tr} 1_Q(x)=\langle \lambda_Q^{b_1}(x,\cdot )b_1, T^{\tr} 1_Q\rangle =\int_Q T(b_1\lambda_Q^{b_1}(x,\cdot )).\end{equation*}
Since $x\in Q$ (otherwise $\lambda^{b_1}_Q =0$), we have that by definition of $\lambda^{b_1}_Q$, the last expression equals
\begin{equation*}\sum_{Q'\text{ children of } Q } 1_{Q'}(x) \left( \int_Q T(b_1 1_{Q'})\right) \frac{1}{|Q'|} \frac{1}{[b_1]_{Q'}}
 \,-\, \left( \int_Q T(b_1 1_Q)\right) \frac{1}{|Q|}\frac{1}{[b_1]_Q}.\end{equation*}
Since $Q\in \Omega_1$, we have that $b_1 \in q-D\psi A$ on $Q$ and all of its children, so that 
$$|[b_1]_Q|,\,|[b_1]_{Q'}|\geq \delta.$$ Consequently,
\begin{equation*}\begin{split} \| \Delta_Q^{b_1}T^{\tr} 1_Q\| _{L^\infty (Q)}&\leq C \sum_{Q'\text{ children of } Q }\left( \frac{1}{|Q|}\int_Q |T(b_1 1_{Q'})|^2 \right)^{\frac{1}{2}} +\,C \left( \frac{1}{|Q|} \int_Q |T(b_11_Q )|^2\right)^{\frac{1}{2}}\\
&\leq C|Q|^{-\frac{1}{2}} \left(\| T b_1 \|_{L^2(Q)} +\| b_1\|_{L^2(2Q)} +\| Mb_1\|_{L^2(Q)}\right)
\leq C,\end{split}\end{equation*}
where we have used Lemma~\ref{l8.33}, and then estimates \eqref{eq8.21} and \eqref{eq8.22}, in the last two inequalitities. Thus, $\| \Delta^{b_1}_Q T^{\tr} 1_Q \| ^2_{L^2(Q)} \leq C|Q|$, and summation over $Q\in \Omega_{2,\buffer}$ completes the estimate
$$\Sigma_2 \leq C|Q_2|.$$

This leaves $\Sigma_1$. That is, we need to prove 
\begin{equation}\label{eq8.39} \sum_{Q\in \Omega_1 \cap \Omega_2} \| \Delta^{b_1}_Q T^{\tr} 1\|^2_2 \leq C|Q_2|,\end{equation}
where we have replaced $1_{Q_1}$ by $1$ in the definition of $\Sigma_1$. Indeed, the error may be controlled by a well-known argument of Fefferman and Stein \cite{FS}, since $\Delta^{b_1}_Q 1=0$, and the kernel of $T^{\tr} $ obeys \eqref{eq6.3}. Combining \eqref{eq8.39} with our estimates for $\Sigma_2$ and $\Sigma_3$, we obtain \eqref{eq8.38}, and thus also the conclusion of Theorem~\ref{t6.6}.

We now proceed to prove \eqref{eq8.39}. We fix $k$ such that $Q\in \mathbb{D}_k$. We begin
 by observing that for $Q\in \Omega_2 \cap \mathbb{D}_k$, we have that
 \begin{equation*} |\Delta_Q^{b_1}T^{\tr}1|\leq \frac{1}{\delta} |(\Delta^{b_1}_Q T^{\tr}1)[b_2]_Q|
 = \frac{1}{\delta} |\Delta_Q^{b_1} T^{\tr} 1E_k b_2|\end{equation*}
 where in the last step we have used that $\Delta_Q^{b_1}T^{\tr}1$ is supported in $Q$, by definition of $\Delta^{b_1}_Q$. We now use a variant of a trick of Coifman and Meyer \cite{CM}, to write
\begin{equation}\begin{split}\label{eq8.40} (\Delta _Q^{b_1}T^{\tr} 1) E_k &= \left\{ (\Delta _Q^{b_1}T^{\tr} 1)E_k -\Delta^{b_1}_QT^{\tr} E_k\right\} +\Delta _Q^{b_1} T^{\tr} (E_k-I)+ \Delta^{b_1}_Q T^{\tr}\\
&\equiv T_{Q,1} +T_{Q,2}+T_{Q,3}.\end{split}\end{equation}
 It is therefore enough to establish \eqref{eq8.39} with $\Delta_Q^{b_1} T^{\tr}1$ replaced by each of $T_{Q,1}b_2$, $T_{Q,2}b_2$ and $T_{Q,3}b_2$.
 
 The contribution of the latter term is easy to handle. To this end, we define an operator $\Lambda_Q^{b_1}$ by the relationship
 $$\Lambda^{b_1}_Q (b_1g)\equiv \Delta^{b_1}_Q g,$$
 i.e. if $\varphi^{b_1}_Q (x,y)$ denotes the kernel of $\Delta^{b_1}_Q$, and, as above
 $$\varphi^{b_1}_Q (x,y)=\lambda^{b_1}_Q (x,y)\,b_1(y),$$
 then $$\Lambda_Q^{b_1}g(x)=\int \lambda^{b_1}_Q (x,y)g(y)\,dy.$$
 We shall prove the following.
 
 \begin{lemma}\label{l8.41} Suppose that $Q_2 \subseteq Q_1$.
Let $b_1,b_2 \in q -D\psi A$ on dyadic sawtooth regions $\Omega_1 \cup \Omega_{\buffer}, \Omega_2 \cup \Omega_{2,\buffer}$, respectively.
Define $C_2\equiv \sup_{Q\in \Omega_2\cup \Omega_{2,\buffer}}[|b_2|^2]_Q$. Then 
 $$\sum_{Q\in \Omega_2\cap \Omega_1} \| \Lambda^{b_1}_Q (b_2g)\| ^2_2 \leq CC_2 \| g\|^2_{L^2(Q_2)}.$$ \end{lemma}
 
 We momentarily defer the proof of Lemma~\ref{l8.41}.
 
 Applying this lemma with $b_1=b_2$, $\Omega_1=\Omega_2$, we obtain
 \begin{equation}\label{eq8.42} \sum_{Q\in \Omega_1} \| \Delta^{b_1}_Q g\|^2_2 \leq C\| g\|^2_2.\end{equation}
 Thus,
 \begin{equation*}\sum_{Q\in \Omega_1\cap \Omega_2} \| T_{Q,3} b_2 \|^2_2 \leq \sum_{Q\in \Omega_1} \| \Delta_Q^{b_1} (1_{Q_2}T^{\tr} b_2)\|^2_2
 \leq C\int_{Q_2} |T^{\tr}b_2|^2\leq C|Q_2|,\end{equation*}
 as desired, where in the last step we have used hypothesis (ii) of Theorem~\ref{t6.6}.
 
 Let us  now prove Lemma~\ref{l8.41}. By \eqref{eq8.20}, and Lebesque's Differentiation Theorem, $b_2\in L^\infty (F_2)$, where $F_2\equiv Q_2\backslash (\cup P^2_i)$, with
 $$\| b_2 \|^2_{L^\infty (F_2)}\leq C_2.$$
 We decompose
 \begin{equation}\label{eq8.43} b_2 g=b_2g1_{F_2} +\sum_i (b_2g)1_{P^2_i}.\end{equation}
 By definition, for $Q\in \mathbb{D}_k$, and $x\in Q$,
 \begin{equation*} \Lambda^{b_1}_Q h(x) 
 =\frac{E_{k+1} h(x)}{E_{k+1} b_1 (x)}-\frac{E_k h(x)}{E_kb_1 (x)}
=\frac{\Delta_k h(x)}{E_{k+1}b_1(x)} -\frac{E_k h(x)\,\Delta_k b(x)}{E_{k+1} b_1 (x)\,E_kb_1(x)}.
\end{equation*}
 Since $Q\in \Omega_1$, we have that $|E_{k+1}b_1(x)|$, $|E_kb_1(x)|\geq \delta$, so by a familiar argument involving \eqref{eq8.2}, Carleson's Lemma and Lemma~\ref{l8.15}, we have that
 \begin{equation}\label{eq8.44} \sum_{Q\in \Omega_1}\| \Lambda^{b_1}_Q h\|^2_2\leq C\| h\|^2_2.\end{equation}
 Consequently
 \begin{equation*} \sum_{Q\in \Omega_2\cap \Omega_1} \| \Lambda^{b_1}_Q (b_2g1_{F_2})\|^2_2 
 \leq C\int_{F_2} |b_2g|^2\leq CC_2\| g\|^2_{L^2(Q_2)}.\end{equation*}
 To treat the second term in \eqref{eq8.43}, we note that if $P^2_i \subseteq Q\in \Omega_2$, then $P^2_i \subsetneq Q$, so that $\lambda^{b_1}_Q (x,y)$ is constant on $P^2_i$. Also,
 $$\int_{P^2_i} (b_2g-[b_2g]_{P^2_i})=0,$$ so therefore we may replace $\sum_i(b_2g)1_{P^2_i}$ by $\sum_i [b_2g]_{P^2_i} 1_{P^2_i}$. This leads to 
 \begin{multline*} \sum_{Q\in \Omega_2\cap \Omega_1} \left\| \Lambda_Q^{b_1} \left( \sum_i 1_{P^2_i} [b_2 g]_{P^2_i}\right) \right\|^2_{L^2}
  \leq C\left\| \sum_i1_{P^2_i}[b_2g]_{P^2_i}\right\|^2_{L^2(Q_2)}\\
 = C\sum_i |P^2_i|[b_2 g]^2_{P^2_i} \leq CC_2 \sum_i \int_{P^2_i}|g|^2,\end{multline*}
 where we have used \eqref{eq8.44} and then Cauchy-Schwarz and the estimate
 $$\frac{1}{|P^2_i|} \int_{P^2_i} |b_2|^2\leq \frac{C}{|2_DP^2_i|} \int_{2_DP^2_i} |b_2|^2 \leq CC_2 .$$
In turn, the latter bound holds because $2_DP^2_i$, the dyadic double of $P^2_i$, belongs to $\Omega_2\cup \Omega_{2,\buffer}$. This concludes the proof of Lemma~\ref{l8.41}, and hence also our treatment of the term $T_{Q,3}$ in \eqref{eq8.40}.
 
 Next, we consider the term $T_{Q,1}$ in \eqref{eq8.40}. By definition, for $Q\in \mathbb{D}_k$
$$1_QE_kb_2=1_Q [b_2]_Q.$$
Thus, since for any $g$, $\Delta_Q^{b_1}g$ is supported in $Q$, we have
\begin{equation*} T_{Q,1}b_2=\Delta_Q^{b_1}T^{\tr} \left(1_{Q^c}([b_2]_Q-E_kb_2)\right)
\equiv T'_{Q,1} b_2+T''_{Q,1}b_2,\end{equation*}  
 where
 $$T_{Q,1}'b_2=\Delta^{b_1}_Q T^{\tr} \left(1_{3Q\backslash Q} ([b_2]_Q -E_kb_2)\right),\,\, \quad
 T''_{Q,1}b_2=\Delta ^{b_1}_Q T^{\tr} \left(1_{(3Q)^c}([b_2]_Q-E_kb_2)\right).$$
 Now, for $Q\in \Omega_1$, $\Delta^{b_1}_Q : L^2(Q)\to L^2(Q)$. Moreover, $T^{\tr}:L^2(3Q\backslash Q) \to L^2(Q)$, by \eqref{eq6.13}. Thus
 \begin{equation*} \| T'_{Q,1} b_2\|_2 \leq C\| [b_2]_Q -E_kb_2\|^2_{L^2(3Q\backslash Q)}
 \leq C\sum^{3^n-1}_{m=1} \| \tilde{\Delta}^m_k b_2\|^2_{L^2(Q)},\end{equation*}
 where $\tilde{\Delta}^m_k$ is defined as follows. Given $Q\in \mathbb{D}_k$, we enumerate the $3^n-1$ cubes in $\mathbb{D}_k$ which are adjacent to $Q$ (i.e., which are contained in $3Q\backslash Q$), and we do this in some canonical fashion so that the enumeration does not depend upon $Q$, but only on position relative to $Q$. Then for any $x\in Q$, and for $Q^m$ one of these enumerated neighbors of $Q$, we set
 $$\tilde{\Delta}^m_kg(x)=[g]_Q -[g]_{Q^m}\equiv \tilde{\Delta} _Q^m g(x)$$
 We leave it to the reader to verify that for each $m=1,2,3,\dots 3^n-1$, we have the square function estimate
 $$\sum_{Q\,\,\text{dyadic}} \| \tilde{\Delta }^m_Q g\|^2_2 =\sum^\infty _{k=-\infty} \| \tilde{\Delta}^m_kg\|^2_2 \leq C\| g\|^2_2.$$
 Consequently,
 $$\sum_{Q\in \Omega_1\cap \Omega_2} \| T'_{Q,1} b_2\|^2_2\leq C\| b_2\|^2_2 =C\| b_2\|_{L^2(Q_2)}^2\leq C|Q_2|.$$
 
 We now turn to the term $T''_{Q,1}b_2$. Let $\psi_Q (x,z)$ denote the kernel of $\Delta ^{b_1}_Q T^{\tr}$. Since $\Delta^{b_1}_Q 1=0$, we have that for $Q\in \Omega_1$ and $z\in (3Q)^c$,
 \begin{multline*} |\psi _Q (x,z)|= \left| \int \varphi^{b_1}_Q (x,y)[K^{\tr} (y,z)-K^{\tr} (x,z) ]dy\right|\\
\leq C1_Q (x) 1_{(3Q)^c} (z)\frac{(\ell(Q))^\alpha}{|x-z|^{n+\alpha}} \frac{1}{|Q|} \int_Q |b_1|
\leq C1_Q \sum^\infty_{i=1} 2^{-i\alpha} (2^i\ell (Q))^{-n} 1_{2^iQ\backslash 2^{i-1} Q}(z),\end{multline*}
 so that
 \begin{equation}\label{eq8.45} |T''_{Q,1} b_2 |\leq C1_Q \sum^\infty _{i=1} 2^{-i\alpha} \frac{1}{|2^iQ|} \int_{2^iQ} |[b_2]_Q -E_kb_2|.\end{equation}
 We note that the concentric dilate $2^iQ$ is covered by a purely dimensional number of dyadic cubes of the same side length $2^i\ell (Q)=2^{i-k}$, namely the dyadic ancestor $(2_D)^iQ$ (here $2_DQ$ denotes the dyadic double of $Q$), along with its neighbors of the same generation $\mathbb{D}_{k-i}$. Enumerating these neighbors in the same canonical fashion as above (i.e., as in the definition of $\tilde{\Delta}^m_k$), we denote them by $Q^m(i)$, $1\leq m\leq 3^n-1$. We then write
 \begin{equation}\begin{split} \label{eq8.46} [b_2]_Q &= [b_2]_Q -[b_2]_{2_DQ} +[b_2]_{2_DQ} -[b_2]_{(2_D)^2Q} +\dots -[b_2]_{(2_D)^iQ} +[b_2]_{(2_D)^iQ} \\
 &= \sum^i_{\ell =1} \Delta_{k-\ell } b_2 (x)+[b_2]_{(2_D)^iQ},\end{split}\end{equation}
 for any $x\in Q$. Similarly,
 \begin{equation}\label{eq8.47} E_kb_2 = E_kb_2-E_{k-1} b_2 +\dots -E_{k-i} b_2+E_{k-i} b_2
 = \sum^i_{\ell =1} \Delta_{k-\ell } b_2+E_{k-i} b_2.\end{equation}
 By definition, $E_{k-i}b_2 (x)=[b_2]_{Q^m(i)}$, if $x\in Q^m(i)$, and $E_{k-i} b_2 (x)=[b_2]_{(2_D)^iQ}$, if $x\in (2_D)^iQ$. Thus, plugging \eqref{eq8.46} and \eqref{eq8.47} into \eqref{eq8.45}, we obtain that
 $$|T''_{Q,1} b_2|\leq C1_Q \sum^\infty_{i=1} 2^{-i\alpha} \left[ \sum^i_{\ell=1}\left(|\Delta_{k-\ell }b_2 |+M(\Delta_{k-\ell} b_2 )\right)+\sum^{3^n-1}_{m=1} |\tilde{\Delta}_{k-i}^m b_2|\right].$$
 Consequently,
 \begin{multline*}\left( \sum_{Q\in \Omega_1 \cap \Omega_2} \| T''_{Q,1} b_2\|^2_2\right)^{\frac{1}{2}}
 \leq C\sum^\infty_{i=1} 2^{-i\alpha} \sum^i_{\ell =1} \left( \sum_k \| \Delta_{k-\ell } b_2\|^2_2\right)^{\frac{1}{2}}\\
+ \,C\sum^{3^n-1}_{m=1} \sum^\infty _{i=1} 2^{-i\alpha} \left( \sum_k \| \tilde{\Delta}^m_{k-i} b_2\|^2_2 \right)^{\frac{1}{2}}
\leq C\| b_2\|_2\leq C|Q_2|^{\frac{1}{2}}.\end{multline*} This completes our treatment of $T_{Q,1}$ in \eqref{eq8.40}.
 
 It remains now to consider the term $T_{Q,2}$, and this will be a more delicate matter. We note that by \eqref{eq8.4} and the definition of $\Delta_j$,
 $$E_k -I=-\sum^\infty_{j=k}\Delta_j$$
 We therefore have that 
 \begin{multline} \label{eq8.48} T_{Q,2}b_2 = 
 -\Delta^{b_1}_Q T^{\tr} \left[ 1_{Q^c} \sum_{j\geq k} \Delta_j b_2\right]\\ +\,\left\{ \Delta ^{b_1}_Q T^{\tr} \left(1_Q ([b_2]_Q -b_2 )\right)-\Lambda^{b_1} _Q \left(( [b_2]_Q -b_2 )Tb_1\right)\right\} \\ 
 +\,\Lambda ^{b_1}_Q \left(([b_2]_Q-b_2 )Tb_1\right)
 \equiv\ \text{Error}_1 \,+ \,G_Q\,+\,\Phi_Q.\end{multline}
 where we have used that $1_Q E_kb_2 =1_Q [b_2]_Q$.
 
 We first turn our attention to $\text{Error}_1$. We fix
$$\delta \equiv C2^{-j\epsilon }2^{-k(1-\epsilon )},$$
 with $C$ a fixed large number and $\epsilon >0$ to be chosen. For each $\mu >0$, we let $Q_\mu$ denote the ``$\mu$-neighborhood of $Q$", i.e.
 $$Q_\mu \equiv \{ x:\dist (x,Q)<\mu \}.$$
 We also define the $\mu $-ring around $Q$ by
 $$R_\mu \equiv Q_\mu \backslash Q.$$
 We choose a smooth cut-off function $\eta_\delta \in C^\infty _0(Q_{2\delta})$, with $\eta _\delta \equiv 1$ on $Q_\delta$, $\| \nabla \eta_\delta \|_\infty \leq C/\delta$,  and $\supp \nabla \eta_\delta \subseteq R_{2\delta }\backslash R_\delta$. We write
 $$1_{Q^c}=1-\eta_\delta +\eta_\delta -1_Q.$$
 We treat the contribution of $1-\eta_\delta$ first; that is, we consider 
 $$ \text{Error}'_1\equiv -\sum_{j\geq k}\Delta _Q^{b_1} T^{\tr} \left((1-\eta _\delta )\Delta^2_jb_2\right),$$
 where we have used that $\Delta_j\equiv \Delta^2_j$. We denote by $h(y,v)$ the kernel of the operator 
$H=T^{\tr} (1-\eta_\delta) \Delta_j$; i.e.
$$h(y,v)=\int K^{\tr} (y,z)\left(1-\eta_\delta (z)\right)\varphi _j(z,v)dz,$$
where the kernel  $\varphi_j(z,v)$ of $\Delta _j$ satisfies $\int\varphi_j(z,v)dz=0$ and $\int |\varphi_j(z,v)|dz\leq C$ for every $v$. We set
$$K^{\tr}_\delta (y,z)=K^{\tr}(y,z)\left(1-\eta_\delta (z)\right).$$
Then for $y\in Q$, we have
\begin{equation*}\begin{split} |h(y,v)|&\leq \int_{|y-z|>c\delta ,\, |z-v|\leq C2^{-j}<<\delta} |K^{\tr}_\delta (y,z)- K^{\tr}_\delta (y,v)| \,|\varphi_j(z,v)|dz\\
&\leq C \frac{2^{-j\alpha}}{|y-v|^{n+\alpha} }1_{\{ |y-v|>c\delta \}} \,+\,C\frac{2^{-j}}{\delta} \frac{1}{|y-v|^n} 1_{\{ c\delta <|y-v|< C\ell (Q)\}}\\
&\equiv h'(y,v) +h''(y,v).\end{split}\end{equation*}
We define operators $H'$, $H''$ by
$$H'g(y)\equiv \int h'(y,v)g(v)dv,\quad H''g(y)\equiv \int h''(y,v) g(v)dv.$$
Recall that $j\geq k$ and that $\delta \equiv C2^{-j\epsilon }2^{-k(1-\epsilon )} =C2^{-j\epsilon}\ell (Q)^{1-\epsilon }$, so that
$$|h'(y,v)|\leq C2^{-(j-k)\alpha(1-\epsilon)}\frac{\delta^\alpha}{(\delta +|y-v|)^{n+\alpha}}.$$
Furthermore,
\begin{equation*}\begin{split} |H''g(y)|&\leq C2^{-(j-k) (1-\epsilon)} \delta^{-n}\int_{|y-v|\leq C\ell (Q)} |g|dv\\
&\leq C2^{-(j-k)(1-\epsilon )} \left( \frac{\ell (Q)}{\delta }\right)^n Mg(y)
= C2^{-(j-k)(1-\epsilon -\epsilon n)}Mg(y).\end{split}\end{equation*}
Combining these estimates, we have that for $\epsilon$ chosen small enough, depending only on $n$, that
$$Hg(y)\leq C2^{-(j-k)\beta } Mg(y),$$
for some $\beta >0$. Now, for $Q\in \Omega_1$, we have that $\Delta _Q^{b_1}: L^2(Q)\to L^2 (Q)$. Consequently, 
\begin{equation*}\| \Delta^{b_1}_Q T^{\tr} \left((1-\eta_\delta )\Delta ^2_j b_2\right)\|_2 
=\| \Delta^{b_1}_Q H\Delta_jb_2\|_2
\leq C2^{-(j-k)\beta} \| M\Delta _jb_2\|_{L^2(Q)}.\end{equation*}
Moreover, summing over $Q\in \mathbb{D}_h \cap \Omega_1\cap \Omega_2$, for each fixed $k$ we obtain
$$\sum_{Q\in \mathbb{D}_k\cap \Omega_1\cap \Omega_2} \| \Delta_Q^{b_1} H\Delta_jb_2\|^2_2\leq C2^{-2(j-k)\beta} \| M\Delta_j b_2\|^2_{L^2(\mathbb{R}^n)}.$$
Therefore, by a variant of Schur's Lemma, we obtain
$$\sum_{Q\in \Omega_1\cap \Omega_2} \| \error '_1 \| _2^2\leq C\| b_2\|^2_2\leq C|Q_2|,$$
as desired.

We now consider the rest of $\error_1$, namely, $$\error ''_1 \equiv -\sum_{j\geq k} \Delta_Q ^{b_1} T^{\tr} \left((\eta_\delta -1_Q)\Delta_j b_2\right).$$
By \eqref{eq6.13},  $T^{\tr} :L^p(6Q\backslash Q)\to L^p(Q)$, $1<p<\infty$. 
We choose $p$ so that $\frac{1}{p}+\frac{1}{q}=1$, where $q$ is the exponent in hypothesis (i) of Theorem~\ref{t6.6}. Then, by definition of $\Delta^{b_1}_Q$, we have that for $Q\in \Omega_1$,
\begin{equation*}\begin{split} |\Delta _Q^{b_1} T^{\tr} \left((\eta_\delta -1_Q)\Delta_j b_2 \right)|&\leq C\frac{1}{|Q|} \int_Q |b_1| \,|T^{\tr} \left((\eta_\delta -1_Q)\Delta_j b_2\right)|\\
&\leq C[|b_1|^q]_Q^{\frac{1}{q}} \left( \frac{1}{|Q|} \int_Q |T^{\tr} \left((\eta_\delta -1_Q)\Delta_j b_2\right)|^p\right) ^{\frac{1}{p}}\\
&\leq C\left( \frac{1}{|Q|} \int_{R_{2\delta}} |\Delta_jb_2|^p\right)^{\frac{1}{p}}\\
&\leq C\left( \frac{|R_{2\delta}|}{|Q|} \right)^{\frac{r-p}{pr}} \left( \frac{1}{|Q|} \int_{2 Q} |\Delta_jb_2|^r\right) ^{\frac{1}{r}}\\
&\leq C2^{-(j-k)\beta} \left(M(|\Delta_jb_2|^r)\right)^{\frac{1}{r}}(x),\end{split}\end{equation*}
for some $\beta >0$, and for all $x\in Q$, where we have used \eqref{eq8.20} in the third inequality, and where $p<r<2$. Thus,
$$\| \Delta_Q^{b_1} T^{\tr} \left((\eta_\delta -1_Q)\Delta_j b_2\right)\|_2 \leq C2^{-(j-k)\beta} \| \left(M(|\Delta_jb_2|^r)\right)^{\frac{1}{r}} \|_{L^2(Q)},$$
so as above we obtain via Schur's Lemma that
$$ \sum_{Q\in \Omega_1 \cap \Omega_2} \| \error ''_1 \|^2_2\leq C\| b_2 \|^2_2 \leq C|Q_2|.$$
This completes our treatment of $\error _1$.

Next, we discuss $\Phi_Q$ in \eqref{eq8.48}. Since $[b_2]_Q\leq C_2$, for all $Q\in \Omega_2$, we have by \eqref{eq8.44} that
\begin{equation*} \sum_{Q\in \Omega_1 \cap \Omega_2} \| \Lambda_Q^{b_1} \left([b_2]_Q Tb_1\right)\|^2_2\leq CC_2\int_{Q_2} |Tb_1|^2
\leq CC_2 |Q_2|,\end{equation*}
where in the last step we have used that the left hand side is zero unless $Q_2\in \Omega_1\cup \Omega_{\buffer}$, so that \eqref{eq8.22} applies to $Tb_1$ in $Q_2$. Moreover, the remaining part of $\Phi_Q$, namely $-\Lambda_Q^{b_1}(b_2Tb_1)$, may be handled similarly via Lemma~\ref{l8.41}. We omit the routine details.

It remains now to treat $G_Q$ in \eqref{eq8.48}. To this end, we set
$$g_Q\equiv 1_Q ([b_2]_Q-b_2),$$
so that $$G_Q= \Delta_Q^{b_1} T^{\tr} g_Q -\Lambda_Q^{b_1} (g_Q\,Tb_1).$$
Suppose that $Q\in \mathbb{D}_k$. We write
\begin{equation*}\begin{split} G_Q&=\left\{ \Delta_Q^{b_1} T^{\tr} E_{k+1} g_Q -\Lambda_Q^{b_1} (E_{k+1}g_Q\,Tb_1)\right\}\\& \qquad+\, \left\{ \Delta^{b_1}_Q T^{\tr} (g_Q-E_{k+1}g_Q)-\Lambda_Q^{b_1}\left((g_Q-E_{k+1}g_Q)Tb_1\right)\right\}\\
&\equiv G'_Q\,+\,\error_2.\end{split}\end{equation*}
We consider $G'_Q$ first. Since $Q\in \mathbb{D}_k$, we have that $E_kg_Q=0$. Thus,
$$E_{k+1}g_Q=(E_{k+1}-E_k)g_Q=\Delta_kg_Q=-\Delta_Qb_2,$$
because $g_Q$ is supported in $Q$, and $\Delta_Q 1=0$. We therefore have that
\begin{equation*}\begin{split} G'_Q &=-\Delta^{b_1}_Q T^{\tr} \Delta_Qb_2+\Lambda^{b_1}_Q \left((\Delta_Q b_2)(Tb_1)\right)\\
&=\I_Q+\II _Q,\end{split}\end{equation*}
and we treat these terms separately. Since $\Delta_Q^{b_1}f=\Lambda_Q^{b_1}(b_1f)=\langle \lambda_Q^{b_1}b_1,f\rangle$, we have that
\begin{equation*}\I_Q(x)= \langle \lambda_Q^{b_1} (x,\cdot )b_1,T^{\tr}(\Delta_Qb_2)\rangle
=\langle T(\lambda^{b_1}_Q (x,\cdot )b_1),\Delta_Qb_2\rangle .\end{equation*}
We recall that by definition
$$\lambda_Q^{b_1}(x,y)=\sum_{Q'} \frac{1}{[b_1]_{Q'}} \frac{1}{|Q'|}1_{Q'} (x)1_{Q'}(y)\,-\,\frac{1}{[b_1]_Q} \frac{1}{|Q|} 1_Q(x)1_Q(y),$$
where the sum runs over the children $Q'$ of $Q$. Thus, 
$$ |\I_Q(x)|\leq C1_Q(x)\left(|Q|^{-1}|\langle T(  1_Q b_1), \Delta_Q b_2\rangle |+\sum_{Q'}|Q'|^{-1} | \langle T(  1_{Q'}b_1) ,\Delta_Q b_2\rangle |\right),$$
where we have used that $Q\in \Omega_1$ to control $[b_1]_Q$ and $[b_1]_{Q'}$ from below (again, the sum runs over the children $Q'$ of $Q$).  But by Cauchy-Schwarz, Lemma~\ref{l8.33}, and \eqref{eq8.21} and \eqref{eq8.22}, this last expression is no longer that
$$C\left( \frac{1}{|Q|} \int_Q |\Delta_Q b_2|^2\right)^{\frac{1}{2}}.$$
Similarly, but more simply, the term $\II_Q(x)$ is dominated by
\begin{equation*}C\left( \frac{1}{|Q|}\int_Q |\Delta_Q b_2|^2\right)^{\frac{1}{2}} 
\left( \frac{1}{|Q|} \int_Q |Tb_1|^2\right)^{\frac{1}{2}}
 \leq C\left( \frac{1}{|Q|} \int_Q |\Delta_Qb_2|^2\right)^{\frac{1}{2}},\end{equation*}
by \eqref{eq8.22}. Altogether then,
$$ \sum_{Q\in \Omega_1\cap \Omega_2} \| G'_Q \|^2_2 \leq \sum_Q \| \Delta_Qb_2\|^2_2 \leq C\|b_2\|^2_2\leq C|Q_2|,$$ as desired.

Finally, we consider the term $\error_2$. For $Q\in \mathbb{D}_k$, the children $Q'$ of $Q$ belong to $\mathbb{D}_{k+1}$, so that for each such child $Q'$, $$\int_{Q'} (g_Q-E_{k+1}g_Q)=0.$$  We set $g'_Q\equiv g_Q-E_{k+1}g_Q$. Now $\Delta^{b_1}_Q f=\Lambda ^{b_1}_Q (b_1f)$, so that for $x\in Q$, we have
\begin{equation*}\begin{split} \error_2 (x)&=\Lambda^{b_1}_Q (b_1 T^{\tr} g'_Q)(x)-\Lambda^{b_1}_Q (g'_Q Tb_1)(x)\\
&= \langle \lambda^{b_1}_Q (x,\cdot ) b_1,T^{\tr} g'_Q\rangle -\langle \lambda_Q^{b_1} (x,\cdot ), g'_Q Tb_1\rangle\\
&= \langle T(\lambda^{b_1}_Q (x,\cdot )b_1),g'_Q\rangle -\langle Tb_1,\lambda_Q^{b_1} (x,\cdot )g'_Q\rangle \\
&=\sum_{Q'} \left(\langle T(\lambda^{b_1}_Q (x,\cdot )b_1),g'_Q1_{Q'}\rangle -\langle Tb_1,\lambda^{b_1}_Q (x,\cdot )g'_Q 1_{Q'}\rangle \right)\\
&=\sum_{Q'} \left(\langle T(\lambda_Q^{b_1}(x,\cdot )b_1) ,g'_Q 1_{Q'}\rangle -\langle T(1_{Q'}\lambda_Q^{b_1} (x,\cdot ) b_1),g'_Q 1_{Q'}\rangle \right),\end{split}\end{equation*} where the sum runs over the children $Q'$ of $Q$, and where, in the last step, we have used that $\lambda^{b_1}_Q (x,\cdot )$ is constant on each child $Q'$ of $Q$. Thus,
\begin{equation*}\begin{split} \error_2 '(x)&=\sum_{Q'} \left\langle 1_{(Q')^C}\lambda_Q^{b_1}(x,\cdot ) b_1,T^{\tr} (g'_Q1_{Q'})\right\rangle \\
&= \sum_{Q'} \Delta_Q^{b_1} \left(1_{Q\backslash Q'} T^{\tr} (g'_Q 1_{Q'} )\right)(x).\end{split}\end{equation*}
Now, by definition,
\begin{equation*}\begin{split} g'_Q 1_{Q'}&=(g_Q-E_{k+1}g_Q)1_{Q'}\\
&= \left\{1_Q ([b_2]_Q -b_2)-E_{k +1} \left(1_Q ([b_2] _Q-b_2)\right)\right\}1_{Q'}\\
&= (E_{k+1}b_2-b_2)1_{Q'},\end{split}\end{equation*}
since $1_{Q'} E_{k+1} (1_Q[b_2]_Q)=[b_2]_Q 1_{Q'}$, for each child of $Q'$ of $Q$. We expand
$$b_2=\sum_{j\geq k} \Delta_jb_2+E_kb_2,$$
and note that since $E_{k+1}E_k-E_k=0$, we have that
$$E_{k+1}b_2-b_2=\sum^\infty_{j=k}(E_{k+1}\Delta_jb_2-\Delta_jb_2).$$
Moreover,
$$E_{k+1}\Delta_j=E_{k+1}(E_{j+1}-E_j)=\begin{cases} E_{k+1}-E_{k+1}=0,&\text{if } j\geq k+1\\
E_{k+1}-E_k =\Delta_k, & \text{if } j=k.\end{cases}$$
Thus,
$$E_{k+1}b_2-b_2=-\sum^\infty_{j=k+1} \Delta_jb_2$$ and consequently,
$$\error _2=-\sum_{Q'\text{ children of } Q} \Delta^{b_1}_Q \left(1_{Q\backslash Q'} T^{\tr} \left(\sum^\infty_{j=k+1} \Delta_jb_21_{Q'}\right)\right).$$
Since $\Delta_j=\Delta_j^2$, it again suffices to show that, for some $\beta >0$, we have
\begin{equation}\label{eq8.49} \| \Delta^{b_1}_Q \left(1_{Q\backslash Q'} T^{\tr} (1_{Q'}\Delta_jh)\right)\|_2 \leq C2^{-\beta (j-k)}\|h\|_{L^2(Q')},\end{equation}
for every $j>k$ and each child $Q'$ of $Q$. Now, the kernel of $1_{Q'}\Delta_j$ is a sum
$$\sum_\ell \varphi_{Q^j_\ell}(z,v),$$
where $\ell(Q^j_\ell )=2^{-j}$, $Q^j_\ell \subseteq Q'$,
\begin{equation}\label{eq8.50} |\varphi_{Q^j_\ell }(z,v)|\leq \frac{C}{|Q^j_\ell |} 1_{Q^j_\ell} (z)1_{Q^j_\ell }(v),\end{equation}
and
\begin{equation}\label{eq8.51} \int\varphi_{Q^j_\ell}(z,v)dz=0,\end{equation}
for each fixed $v$. We split
$$1_{Q\backslash Q'} =1_{Q\backslash Q'_\delta} 
+1_{ R'_\delta \cap Q},$$
where as before $Q'_\delta$ is the $\delta$ neighborhood of $Q'$, and $R'_\delta = Q'_\delta \backslash Q'$. In the present situation, we choose $\delta=2^{-j/2}2^{-k/2}$. We let $J(y,v)$ denote the kernel of $T^{\tr} 1_{Q'}\Delta _j$, and observer that for $y\in Q\backslash Q'_\delta$, and by \eqref{eq8.50} and \eqref{eq8.51}, we have
\begin{equation*}\begin{split} |J(y,v)|&=\left| \sum_\ell 1_{Q^j_\ell } (v)\int \left(K^{\tr} (y,z)-K^{\tr} (y,v)\right) \varphi_{Q^j_\ell} (z,v)dz\right| \\
&\leq C\sum_\ell 1_{Q^j_\ell } (v)\int\frac{2^{-j\alpha}}{|y-v|^{n+\alpha}} \,|\varphi_{Q^j_\ell }(z,v)|\,dz\,1_{\{ |y-v|>C\delta\}}\\
&\leq C1_{Q'}(v)2^{-(j-k)\alpha/2} \frac{\delta^\alpha}{(\delta +|y-v|)^{n+\alpha}}.\end{split}\end{equation*}
Since $\Delta^{b_1}_Q:L^2(Q)\to L^2(Q)$, for $Q\in \Omega_1$, we have that \eqref{eq8.49} holds for the contribution of $1_{Q\backslash Q'_\delta}$.

It remains now only to treat the contribution of $1_{R'_\delta \cap Q}$. To this end, we recall that $\varphi_Q^{b_1}(x,y)$, the kernel of $\Delta^{b_1}_Q$, satisfies
$$|\varphi_Q^{b_1}(x,y)|\leq \frac{C}{|Q|} 1_Q(x)1_Q (y) b_1(y).$$
Then for $x\in Q\in \Omega_1\cap \mathbb{D}_k$, $q$ as in hypothesis (i) of Theorem~\ref{t6.6} (and also \eqref{eq8.20}), $\frac{1}{p}+\frac{1}{q}=1$, and $p<r<2$, we have
\begin{equation*}\begin{split} |\Delta_Q^{b_1}1_{R'_\delta \cap Q} T^{\tr} (1_{Q'} \Delta_jh)(x)|&\leq C\frac{1}{|Q|} \int_{R'_\delta \cap Q} |b_1|\, |T^{\tr}(1_{Q'}\Delta_jh)|\\
&\leq C\left( \frac{1}{|Q|} \int_Q |b_1|^q\right)^{\frac{1}{q}} \left( \frac{1}{|Q|} \int_{R'_\delta \cap Q} |T^{\tr} (1_{Q'} \Delta _jh)|^p \right)^{\frac{1}{p}}\\
&\leq C\left( \frac{|R'_\delta \cap Q|}{|Q|} \right)^{\frac{r-p}{rp}} \left( \frac{1}{|Q|} \int_{Q'_\delta \backslash Q'} |T^{\tr} (1_{Q'} \Delta_jh)|^r\right) ^{\frac{1}{r}}\\
&\leq C2^{-(j-k)\beta} \left( \frac{1}{|Q|} \int_{Q'} |\Delta_jh|^r\right)^{\frac{1}{r}},\end{split}\end{equation*}
for some $\beta >0$, where in the last step we have used the dual of the $L^r$ version of \eqref{eq6.13}. Since $Q'$ is a child of $Q$, the last expression is bounded by
$$C2^{-(j-k)\beta}\left(M(|1_{Q'}\Delta_jh|^r)\right)^{\frac{1}{r}} (x),$$
for every $x\in Q$. Since $1_{Q'} \Delta_j h=1_{Q'}\Delta_j(1_{Q'} h)$, for $j\geq k+1$, \eqref{eq8.49} follows.

\end{document}